\documentclass{article}
\usepackage[utf8]{inputenc}
\usepackage{amsfonts}
\usepackage{amsmath}
\usepackage{amsthm}
\usepackage{comment}
\usepackage{hyperref}
\DeclareSymbolFont{rsfscript}{OMS}{rsfs}{m}{n}
\DeclareSymbolFontAlphabet{\mathrsfs}{rsfscript}
\theoremstyle{plain}
\newtheorem{theorem}{Theorem}

\newtheorem{lemma}[theorem]{Lemma}

\theoremstyle{definition}
\newtheorem{remark}[theorem]{Remark}

\newtheorem{definition}[theorem]{Definition}

\title{Fast synchronization of inhomogenous random automata}
\author{Balázs Gerencsér\thanks{B.\ Gerencs\'er is with the Alfr\'ed R\'enyi Institute of Mathematics, Budapest, Hungary and E\"otv\"os Lor\'and University, Department of Probability and Statistics, Budapest, Hungary, {\tt\small gerencser.balazs@renyi.hu} He was supported by NRDI (National Research, Development and Innovation Office) grant KKP 137490 and the J\'anos Bolyai Research Scholarship of the Hungarian Academy of Sciences.}, %
\space Zsombor Várkonyi\thanks{Zs.\ Várkonyi is with the E\"otv\"os Lor\'and University, Budapest, Hungary, {\tt\small zsvarkonyi@student.elte.hu} Supported by the ÚNKP-21-6-ELTE-1153 New National Excellence Program of the Ministry for Innovation and Technology from the source of the NRDI Fund.}}
\date{\today}

\begin{document}
\maketitle
\begin{abstract}
    We examine the reset threshold of randomly generated deterministic automata. We present a simple proof that an automaton with a random mapping and two random permutation letters has a reset threshold of $\mathcal{O}\big( \sqrt{n \log^3 n} \big)$ with high probability, assuming only certain partial independence of the letters. Our observation is motivated by Nicaud \cite{nicaud2019vcerny} providing a near-linear bound in the case of two random mapping letters, among multiple other results. 
    The upper bound for the latter case has been recently improved by the breakthrough work of Chapuy and Perarnau \cite{chapuy2023short} to $\mathcal{O}(\sqrt{n} \log n)$.
\end{abstract}



\section{Introduction}
The synchronization of automata has been intensively studied in the literature since the sixties. Recently, a related area has been spotlighted, namely synchronizing random automata, viewing automata theory in a probabilistic way.

For a complete finite automaton, a word is called a \textit{synchronizing word} if it maps to a state independent of the initial state. The \textit{reset threshold} of an automaton with a synchronizing word is the length of its shortest synchronizing word. In 1964, Černý conjectured \cite{vcerny1964poznamka} that if an $n$-state automaton can be synchronized, then its reset threshold is at most $(n-1)^2$. This is an open question ever since, however some partial results were proven over the years. Firstly, this is a strict bound, meaning that for all $n\in \mathbb{N}$ there is an $n$-state automaton with reset threshold exactly $(n-1)^2$. The best known upper bound was attained by Shitov \cite{shitovimprovement} which improved the previous upper bound of Szykula \cite{szykula2018improving}. Both of these results provide a bound of $\Theta (n^3)$ and only differ by the leading coefficient. The first cubic upper bound with leading coefficient near $1/6$ was reached by Pin and Frankl \cite{FRANKL1982125}, \cite{pin1983two}.

A strongly related area of study is the investigation of randomly generated deterministic automata, in which the aim is to prove statements that hold with high probability for a certain class of automata. Most of these works concern automata with independent letters.

On the analytical side, Nicaud \cite{nicaud2019vcerny} showed that an automaton with at least two random mapping letters in its alphabet has reset threshold of $\mathcal{O}(n \log^3 n)$ with high probability. Recent related results include the paper of Catalano and Jungers \cite{catalano2018randomized} in which the positivity of exponents of quasi-permutation matrices is examined and the article of Chapuy and Perarnau \cite{chapuy2023short} in which an $\mathcal{O}(\sqrt{n} \log n)$ upper bound is proven for random automata with two random mappings.

From a numerical perspective, Skvortsov and Tipikin performed an experimental study \cite{skvortsov2011experimental} and conjectured that the expected value of the reset threshold of an $n$-state finite automaton is sublinear, more precisely $\mathcal{O}(n^{0.55})$. Kisielewicz et al.\ examined \cite{kisielewicz2013fast} a large number of random automata up to at most 300 states and conjectured that the expected reset threshold of a random automata with two random mappings is $\mathcal{O}(\sqrt{n})$. These results show that the upper bound of \cite{chapuy2023short} is expected to be tight up to logarithmic factors. 

In this paper, we are going to examine a class of letter-inhomogenous random automata and prove that for an $n$-state automaton of this class, its reset threshold is $ \mathcal{O} \big( \sqrt{n \log^3 n} \big)$ with high probability. This result remains interesting on its own even in view of a possible worst-case quadratic bound of the Černý conjecture as the target high probability bound is substantially lower. 

\section{Definitions and Notations}

A deterministic finite automaton is denoted by $\mathrsfs{A}=(Q,\Sigma, \delta)$ where the set of states is $Q$, the set of letters is $\Sigma$ and the transition funcion of letters over the states is $\delta : Q \times \Sigma \rightarrow Q$. A \textit{word} is a finite series of not necessarily distinct letters, formally words are the elements of $\Sigma ^ k$ for arbitrary $k\in \mathbb{N}$. The transition function $\delta$ extends to words in a natural way. Automaton $\mathrsfs{A}$ is called \textit{synchronizing} if a word $u$ exists such that the action of word $u$ leaves the automaton in a state independent of the initial state, that is $\exists S \in Q , \forall T \in Q: \; \; S=\delta(T, u)$. A word $u$ with this property is called a \textit{synchronizing word}. A series of events $(A_n)_{n \in \mathbb{N}}$ is said to hold \emph{with high probability} if $P(A_n) \longrightarrow 1$.


A randomly generated deterministic automaton can be built up by random letters of different nature. Let us present the different types of random letters we will be using.

\begin{definition}\label{random_mapping_pmapping}
Let us define \emph{random mapping}, \emph{random p-mapping} and \emph{random permutation} as possible letters of a finite automaton. Let $Q=\{1,2,...n\}$.
\begin{itemize}
    \item $f : Q \rightarrow Q$ is called a \emph{random mapping} if $f(i)$ is chosen uniformly randomly from $Q$, independently for all $i=1,2,...n$.
    
    \item  Given a probability mass function $p$ on $Q$, $g : Q \rightarrow Q$ is a \textit{random p-mapping} if $g(i)$ is chosen from $Q$ according to the probability $p$ independently for all $i=1,2,...n$.
    
    \item $h : Q \rightarrow Q$ is a \emph{random permutation} if it is chosen uniformly randomly from the permutation group $S_n$.
\end{itemize}
\end{definition}

Note that an automaton with random mapping, random p-mapping or random permutation letters still has deterministic actions for every realisation of the letters.

Considering the structure of random letters, we introduce the following concepts. For a random mapping $f$ a state $S$ is a \emph{cyclic point} if $f^i(S)=S$ for some $i>0$. The \emph{height} of a state $T$ is the smallest $i\ge 0$ such that $f^i(T)$ is a cyclic point. The height of $f$ is maximal height among all states.

\section{Tools}
Before presenting our main result, let us cite two important lemmas which will be useful for us. The first one is about random mappings from Nicaud \cite{nicaud2019vcerny}: 
\begin{lemma}[Lemma 4 in \cite{nicaud2019vcerny}] \label{Nicaud_lemma}
The probability that a random $p$-mapping of size $n$ has more than $2\sqrt{n\log n}$ cyclic points or that it has height greater than $2\sqrt{n\log n}$ is $\mathcal{O}(\frac{1}{n})$.
\end{lemma} 
In \cite{nicaud2019vcerny} this lemma was used to prove that that a random automaton with at least two random mapping letters in its alphabet has a reset threshold at most $\mathcal{O}(n \log^3 n)$ with high probability. Note that by repeating the same random p-mapping letter several times, the number of possible states can be significantly decreased. 

In their paper \cite{friedman1998action}, Friedman and his coauthors proved a general statement about the transitivity of random permutations. Their result is the following:
\begin{lemma}\label{permutations_lemma}
For every $r$ and $d\ge 2$ there is a $C$ such that for uniformly random choices of $d$ permutations $\pi_1, \pi_2, \dots , \pi_d$ of $S_n$, the following holds with high probability: for any two $r$-tuples of distinct elements in $\{ 1,\dots ,n\}$, there is a product of at most $C \log n$ of the given permutations which map the first $r$-tuple to the second.
\end{lemma}

\section{Main Result}

\begin{theorem}
Let $\mathrsfs{A}$ be a randomly generated deterministic complete automaton with $|Q| = n$ and $\Sigma = \{ a, b, c\}$. Let $a$ be a random mapping, $b$ and $c$ be random permutations such that $b$ and $c$ are independent. The reset threshold of $\mathrsfs{A}$ is $\mathcal{O}\big( \sqrt{n \log^3 n} \big)$ with high probability.
\end{theorem}

\begin{remark}
Let us highlight the most important differences between the model we investigate and the usual model of independent uniform random mappings. The independence of the two random permutations guarantee that our automaton is strongly connected with high probability \cite{friedman1998action}, moreover generates $S_n$ or $A_n$ with high probability \cite{dixon1969probability}, neither of which holds for an automaton with three random mappings. 
The idea of using permutations is motivated by the work of Nicaud and Berlinkov \cite{berlinkov2020synchronizing}.
\end{remark}

\begin{proof}
Concerning Lemma \ref{permutations_lemma}, let us fix $d=r=2$, with the resulting constant being $C$. 
Let $R_n$ denote the event that both the conclusions of Lemma \ref{Nicaud_lemma} and Lemma \ref{permutations_lemma} hold true for our automaton. 
Since both conclusions hold with high probability, it follows that $R_n$ holds with high probability as well.
For simplicity, let us use the notation ${H}(n)=2\sqrt{n\log n}$. 

First we apply $a^{H(n)}$ to the entire set $Q$. By Lemma \ref{Nicaud_lemma}, after applying $a^{H(n)}$, the maximal number of possible states is $H(n)$ on $R_n$, from which we proceed as follows.
Thanks to $a$ contracting the original set of possible states, we may choose two states ($P_1$ and $P_2$) that are being merged by $a$.
Formally, $\exists P_1, P_2 \in Q : \delta(P_1,a) = \delta(P_2,a)$.

We repeat the following procedure until the automaton is synchronized:

\begin{itemize}
    \item We search for two different possible states of the automaton. If there is only one, then the automaton is synchronized and we are done. Let these two different states be named $T_1, T_2$.
    \item Now we will map $(T_1,T_2)$ pair to $(P_1,P_2)$ pair. By Lemma \ref{permutations_lemma} on $R_n$, we can do this step in $C \log n$ transitions.
    \item Now apply letter $a$ and this reduces the number of possible states by at least one, since $P_1$ and $P_2$ are merged by $a$.
\end{itemize}

We have to repeat this process at most $H(n)-1$ times as there were at most $H(n)$ possible states on $R_n$ after applying $a^{H(n)}$. The upper bound for the number of letters needed is: $H(n)+(H(n)-1)(C \log n +1) \sim CH(n)\log n = \mathcal{O}\big( \sqrt{n \log^3 n} \big)$. Here we calculated the length of the word using the idea of Lemma 2 in \cite{nicaud2019vcerny}, namely that pairwise synchronization is sufficient for the entire automaton to be synchronized. 

Note that the subset $\delta(Q,a^{H(n)})$ and the functions $(b,c)$ might not be independent, but we only use the consequences of the lemmas via assuming $R_n$, no further structure is exploited.
As this event holds with high probability, this confirms the claim.
\end{proof}

\bibliographystyle{plain}
\bibliography{references}

\end{document}